\documentclass[12pt]{amsart}
\usepackage{amsmath,amsfonts,amssymb,mathrsfs,amstext,amscd,latexsym,amsthm}
\usepackage{comment}
\usepackage{mathtools}
\usepackage{hyperref}
\usepackage{enumitem}
\usepackage{xcolor}
\newcommand{\R}{\ensuremath{\mathbb{R}}}

\newcommand{\N}{\ensuremath{\mathbb{N}}}

\DeclareMathOperator{\tr}{\mathrm{tr}}

\newcommand{\pd}{\partial}
\newcommand{\cd}{\nabla}
\DeclareMathOperator{\area}{\mathrm{area}} 
\DeclareMathOperator{\Rc}{\mathrm{Rc}} 
\DeclareMathOperator{\Rm}{\mathrm{Rm}} 

\newcommand{\inner}[2]{\left\langle #1 \, , \, #2\right\rangle} 

\def\labelitemi{--}
\def\ba #1\ea {\begin{align} #1\end{align}}
\def\bann #1\eann {\begin{align*} #1\end{align*}}
\def\ben #1\een {\begin{enumerate} #1\end{enumerate}}
\def\bi #1\ei {\begin{itemize}\renewcommand\labelitemi{--} #1\end{itemize}}
\theoremstyle{plain}
\newtheorem{thm}{Theorem}
\newtheorem*{thm*}{Theorem}
\newtheorem{lem}[thm]{Lemma}

\newtheorem{claim}[thm]{Claim}
\newtheorem{prop}[thm]{Proposition}

\newtheorem*{conds*}{Conditions}
\newtheorem*{auxconds*}{Ancillary Conditions}
\newtheorem*{props*}{Properties}
\theoremstyle{remark}
\newtheorem{rem}{Remark}

\usepackage{cite}

\title[Hamilton's Theorem via mean curvature flow]{Hamilton's Theorem (on the compactness of pinched hypersurfaces) via mean curvature flow.}

\author{Theodora Bourni}
\author{Mat Langford}
\author{Stephen Lynch}

\address{Department of Mathematics, University of Tennessee Knoxville, Knoxville TN, 37996-1320}
\email{tbourni@utk.edu}
\email{mlangford@utk.edu}

\address{School of Mathematical and Physical Sciences, The University of Newcastle, Newcastle, NSW, Australia, 2308}
\email{mathew.langford@newcastle.edu.au}

\address{Eberhard Karls Universit\"at T\"ubingen, Fachbereich Mathematik, Auf der Morgenstelle 10, 72076 T\"{u}bingen, Germany}
\email{stephen.lynch@math.uni-tuebingen.de}

\date{\today}

\begin{document}

\begin{abstract}
We make rigorous and old idea of using mean curvature flow to prove a theorem of Richard Hamilton on the compactness of proper hypersurfaces with pinched, bounded curvature.
\end{abstract}

\maketitle

A famous theorem of Myers \cite{Myers} implies that a complete Riemannian manifold with uniformly positive Ricci curvature is necessarily compact. By the Gauss equation, this implies that a properly immersed hypersurface of Euclidean space $\R^{n+1}$, $n\ge 2$, with uniformly positive second fundamental form is compact. 

Hamilton obtained a scale invariant version of this result \cite{HamiltonPinched}: a smooth, proper, \emph{locally} uniformly convex hypersurface $M^n\to \R^{n+1}$, $n\ge 2$, which is \emph{pointwise pinched}, i.e.
\[
\kappa_1\ge \alpha\kappa_n\;\;\text{for some}\;\;\alpha>0,
\]
where $\kappa_1\le\dots\le\kappa_n$ denote the principal curvatures (eigenvalues of the second fundamental form $A$), is necessarily compact. 
Hamilton's argument exploits the fact that the Gauss map of such a hypersurface is quasiconformal.

Inspired by Hamilton's theorem, Ni and Wu \cite{NiWuPinched} (see also \cite{ChZh}) proved an analogous \emph{intrinsic} pinching theorem: any smooth, complete Riemannian manifold $M^n$, $n\ge 3$, with bounded, non-negative curvature operator $\Rm$ which is \emph{pointwise pinched}, i.e.,
\[
\rho_1\ge \alpha\rho_N\;\;\text{for some}\;\;\alpha>0,
\]
where $\rho_1\le \dots\le\rho_N$ are the eigenvalues of $\Rm$, is either flat or compact. This intrinsic result is proved using Ricci flow. The main tools in the argument are the improvement-of-pinching theorem of B\"ohm and Wilking \cite{BohmWilking} and a non-existence result for pinched Ricci solitons.

It has long been believed that Hamilton's theorem can be proved using extrinsic geometric flows such as the mean curvature flow. The idea, which combines many classical ideas of Hamilton, is as follows: suppose, to the contrary, that there exists a non-flat pinched proper hypersurface which is \emph{not} compact. Evolve it by curvature. If the evolving hypersurface becomes singular in finite time, then, due to the uniform pinching, we can blow-up \emph{\`a la} Huisken \cite{Hu84} to obtain a shrinking sphere solution, violating the noncompactness. Else, the flow exists for all time. But then we can blow-up at time infinity \emph{\`a la} Hamilton \cite{HamiltonSingularities} to obtain a non-flat, pinched (translating or expanding) soliton solution. But these may be ruled out directly.

This argument is quite simple and elegant in concept but will require some deep analytic facts about mean curvature flow to make rigorous. Unfortunately, as for the result of Ni and Wu, we require the additional hypothesis of bounded curvature (Hamilton's argument requires no such hypothesis). Removing this hypothesis in the analysis of noncompact solutions to mean curvature flow is a deep issue of independent interest.\footnote{Compare the very recent work of Deruelle--Schulze--Simon \cite{DSSRiccipinched}, Lott \cite{LottRiccipinched}, and Lee--Topping \cite{LTRiccipinched,LTPIC1} on Ricci flow and Daskalopoulos--S\'aez \cite{DaskSaez} and the authors \cite{BLLnoncollapsing} on mean curvature flow.} On the other hand, we only require weak local convexity and our proof appears to be generalizable to other settings. 

The key novel ingredient we will need is the following localization of Huisken's umbilic estimate \cite[5.1 Theorem]{Hu84}, which is an immediate corollary of the local pinching estimate recently proved in \cite{LocalPinching}.

\begin{prop}[Local umbilic estimate]\label{prop:local pinching}
Every mean curvature flow in $\R^{n+1}$ which is properly defined and $\alpha$-pinched in $B_{2L}\times [0,T)$ satisfies
\ba\label{eq:m convexity}
\vert\mathring A\vert
\le \varepsilon H+C_\varepsilon \Theta\;\;\text{in}\;\; B_{L/2}\times[0,T)\,,
\ea
where $\mathring A$ is the trace-free part of $A$, $\Theta\doteqdot \sup_{B_{2L}\times\{0\}\cup B_{2L}\setminus B_{L}\times (0,T)}H$, and $C_\varepsilon=C(n,\alpha,\varepsilon)$.
\end{prop}
\begin{proof}
The claim follows from the $m=0$ case of \cite[Theorem 1]{LocalPinching} since the integral hypothesis is in this case superfluous. Indeed, by the pinching hypothesis and the area formula,
\[
\int_{M_t\cap B_{2L}}H^n\,d\mu\le \left(\frac{n}{\alpha}\right)^n\int_{M_t\cap B_{2L}}K\,d\mu\le \left(\frac{n}{\alpha}\right)^n\operatorname{area}(S^n)
\]
for each $t\in[0,T)$, where $K$ is the Gauss curvature.
\end{proof}

We also require a suitable existence result for a mean curvature flow out of the boundary of a convex body. 
This is a straightforward consequence of the following Chou--Ecker--Huisken-type \cite{Ts85,EckerHuisken91} 
estimate for radial graphs (see, for example, \cite[Corollary 2.2]{LynchUniqueness} for a proof).

\begin{lem}\label{lem:global in time est}
There exists $C=C(n)$ with the following property. Let $\{\pd\Omega_t\}_{t \in [0,T)}$ be a convex solution to mean curvature flow in $\R^{n+1}$. If $B_r(p)\subset\Omega_t$ and $\sup_{B_{2Lr}}H(\cdot,0)\le\Theta r^{-1}$, where $\Theta\ge 1$ and $L>1$, then
\begin{equation*}
\sup_{X\in B_{Lr}(p)\cap \pd\Omega_t}\left(1-\frac{\vert X-p\vert^2}{L^2r^2}\right)H(X,t) \leq CL^3\Theta r^{-1}\,.
\end{equation*}
\end{lem}

\begin{prop}[Existence]\label{prop:existence}
Let $\Omega\subset\R^{n+1}$ be an (unbounded) convex body with smooth boundary $\pd\Omega$ satisfying $\sup_{\pd\Omega}\vert A\vert<\infty$. There exist $\delta>0$ and a family of (unbounded) convex bodies 
whose boundaries $\{\pd\Omega_t\}_{t\in(0,\delta]}$ are smooth and evolve by mean curvature, converge locally uniformly to $\Omega_0\doteqdot\Omega$ as $t\to 0$, and satisfy $\sup_{t\in [0,\delta]}\sup_{\pd\Omega_t}\vert A\vert<\infty$.
\end{prop}
\begin{proof}
We may assume that $\Omega$ contains no lines --- else it splits as a product of an affine subspace with a lower dimensional convex body which contains no lines; the solution we seek is then obtained as a product of an affine subspace with the lower dimensional solution we shall construct. In that case, $\Omega$ is either bounded or, up to a rotation, the graph of a function $u:D\to\R$ over some convex domain $D\subset \R^n\times\{0\}$ with compact sublevel sets $\{u\le h\}$. When $\Omega$ is unbounded, we consider for each height $h>0$ the bounded convex body $\Omega^h$ defined by intersecting $\Omega$ with its reflection about the plane $\{x_{n+1}=h\}$ (if $\Omega$ is bounded, just take $\Omega^h=\Omega$ in what follows). Let $B^h$ be the largest ball contained in $\Omega^h$ 
and denote by $r_h:\pd B^n\to\R$ the radial graph height of $\pd\Omega^h$. We now mollify $\Omega^h$ to obtain, for any sufficiently small $\varepsilon>0$, a smooth convex body $\Omega^{h,\varepsilon}$ with corresponding smooth radial graph function $r_{h,\varepsilon}:\pd B^n\to\R$ (for example, we could mollify the radial graph functions $r_h$ using, say, the heat kernel on $S^n$). 

For $\varepsilon$ sufficiently small, we can evolve $r_{h,\varepsilon}$ smoothly by radial graphical mean curvature flow using parabolic existence theory. Since, by the avoidance principle, the time of existence of the approximating solutions is bounded from below by the square of their inradius (which is bounded uniformly from below as $\varepsilon\to 0$ and $h\to\infty$) Lemma \ref{lem:global in time est} and the higher order estimates of Ecker and Huisken \cite{EckerHuisken91}  yield the claims upon taking $\varepsilon\to 0$ and then $h\to\infty$.
\end{proof}

In order to exploit Proposition \ref{prop:local pinching}, we will need two ingredients. First, we need to preserve the initial pinching condition. 

\begin{prop}[Pinching is preserved]\label{prop:pinching preserved}
Let $\{M_t\}_{t\in[0,\delta]}$ be a family of convex, locally uniformly convex hypersurfaces evolving by mean curvature flow with $\sup_{t\in[0,\delta]}\sup_{M_t}\vert A\vert<\infty$. 
If $\kappa_1\ge \alpha H$ on $M_0$, then
\[
\kappa_1\ge \alpha H\;\;\text{on}\;\; M_t\;\;\text{for all}\;\;t\in[0,\delta]\,.
\]
\end{prop}
\begin{proof}
Since $M_0$ does not contain any lines, we can find $p\in \R^{n+1}$ and $e\in S^n$ so that $\beta\doteqdot \inf_{X\in M_0}\inner{\frac{X-p}{\vert{X-p\vert}}}{e}>0$. If we define
\[
\psi(X)\doteqdot \inner{X-p}{e}\mathrm{e}^{(C+1)t}\,,
\]
where $C\doteqdot\sup_{t\in[0,\delta]}\sup_{M_t}\vert A\vert^2$, then
\[
\psi\ge \beta\vert X-p\vert\;\;\text{and}\;\;(\pd_t-\Delta)\psi=(C+1)\psi\,.
\]

Fix $\varepsilon>0$ and set $S\doteqdot A-\alpha Hg$. We claim that the tensor
\[
S^{\varepsilon}\doteqdot S+\varepsilon\psi g
\]
remains non-negative definite on $[0,\delta]$. 
%
Suppose that this is not the case. Then, since $S$ is positive definite on $M_0$ and $\psi\to\infty$ as $\vert X\vert\to\infty$, there must exist $t_0\in(0,\delta]$, $X_0\in M^n_{t_0}$, and $V_0\in T_{X_0}M^n_{t_0}$ such that $S^{\varepsilon}_{(X,t)}>0$ for each $X\in M^n_t$, $t\in [0,t_0)$, but $S^{\varepsilon}_{(X_0,t_0)}(V_0,V_0)=0$. Extend $V_0$ locally in space by solving 
\[
\cd_{\gamma'} V\equiv 0
\]
along radial $g_{t_0}$-geodesics $\gamma$ emanating from $X_0$ and then extend the resulting local vector field locally in the time direction by solving
\[
\cd_tV\equiv 0\,,
\]
where $\cd_t$ is the time-dependent connection of Andrews and Baker \cite{AnBa10}. We find that $\cd V(X_0,t_0) = 0$, $\cd_t V(X_0,t_0) = 0$, and $\Delta V (X_0,t_0) = 0$. 

Now set $s_{\varepsilon}(X,t)\doteqdot S^{\varepsilon}_{(X,t)}(V_{(X,t)},V_{(X,t)})$ for $(X,t)$ near $(X_0,t_0)$. 
We now find at $(X_0,t_0)$ that
\bann
0\geq(\pd_t-\Delta)s_{\varepsilon}={}&(\cd_t-\Delta)S^{\varepsilon}(V,V)\\
\geq{}&\vert A\vert^2S(V,V)+\varepsilon(C +1)\psi\\
={}&-\varepsilon \psi\vert A\vert^2+\varepsilon(C +1)\psi\\
\geq{}&\varepsilon\psi\\
>{}&0\,,
\eann
which is absurd. So $S^{\varepsilon}$ indeed remains positive definite in $[0,\delta]$. Now take $\varepsilon\to 0$.
\end{proof}

Second, we need a bound for the curvature at infinity which is uniform in time. 

\begin{prop}[Curvature bound at infinity]\label{prop:curvature bound at infinity}
Let $\{\pd\Omega_t\}_{t\in[-\frac{1}{2n}R^2,0]}$ be a family of convex boundaries evolving by mean curvature. Suppose that $0 \in \Omega_0$. Given $\varepsilon>0$ and $\delta>0$, there exists $L<\infty$ such that, given any $X \in M_0 \setminus B_{LR}(0)$,
\[H(X,0) \geq \delta R^{-1}\;\;\implies\;\;\inf_{B_{\frac{\delta}{2H(X,0)}}(X)\times(-\frac{1}{2n}\frac{\delta^2}{4H^2(X,0)},0]}\frac{\kappa_1}{H}\leq \varepsilon\,.\]
\end{prop}
\begin{proof}
It suffices to prove the claim when $R=1$. So defining $P_r(X,t)\doteqdot B_r(X)\times(t-\frac{1}{2n}r^2,t]$ suppose, contrary to the claim, that we can find $\varepsilon>0$, $\delta>0$, and a sequence of points $X_j\in M_0$ such that
\[
\vert X_j\vert\underset{j\to\infty}{\to}\infty\,,\;\; H(X_j,0)\ge \delta\,,\;\;\text{and yet}\;\;\inf_{P_{\frac{\delta}{2H(X_j,0)}(X_j,0)}}\frac{\kappa_1}{H}(X_j,0)>\varepsilon\,.
\]
Point selection 
yields a sequence of points $(Y_j,s_j)$ with the following properties:
\begin{enumerate}
\item \label{eq:pp1} $(Y_j,s_j) \in P_{\frac{\delta}{2H(X_j,0)}}(X_j,0)$ (and hence $\frac{\kappa_1}{H}(Y_j,s_j)>\varepsilon$).
\item \label{eq:pp2} $H(Y_j,s_j) \geq H(X_j,0)$.
\item \label{eq:pp3} $H \leq 2H(Y_j,s_j)$ in $P_{\frac{\delta}{4H(Y_j,s_j)}}(Y_j,s_j)$. 
\end{enumerate}
Indeed, if the choice $(Y_j,s_j)=(X_j,0)$ satisfies \eqref{eq:pp3}, then we take $(Y_j,s_j)=(X_j,0)$. If not, choose $(X_j^1,t_j^1) \in P_{\frac{\delta}{4H(X_j,0)}}(X_j,0)$ such that 
\[
H(X_j^1,t_j^1) \geq 2H(X_j,0)\,.
\]
So $(X_j^1,t_j^1)$ satisfies \eqref{eq:pp1} and \eqref{eq:pp2}. If $(X_j^1,t_j^1)$ also satisfies \eqref{eq:pp3}, choose $(Y_j,s_j)=(X_j^1,t_j^1)$; if not, choose $(X_j^2,t_j^2) \in P_{\frac{\delta}{4H(X_j^1,t_j^1)}}(X_j^1,t_j^1)$ such that
\[
H(X_j^2,t_j^2) \geq 2H(X_j^1,t_j^1)\,.
\]
Since $P_{\frac{\delta}{4H(X_j^1,t_j^1)}}(X_j^1,t_j^1)\subset P_{\frac{\delta}{4H(X_j,0)}+\frac{\delta}{8H(X_j,0)}}(X_j,0)$, this point satisfies properties \eqref{eq:pp1} and \eqref{eq:pp2}. Since $\sum_{k=1}^\infty 2^{-k}=1$ and $H$ is finite on the set $P_{\frac{\delta}{H(X_j,0)}}(X_j,0)$, continuing in this way we find, after some finite number of steps $k$, a point $(Y_j,s_j)\doteqdot (X_j^k,t_j^k)$ satisfying \eqref{eq:pp1}, \eqref{eq:pp2} and \eqref{eq:pp3}.

Now translate $(Y_j,s_j)$ to the spacetime origin and rescale by $\lambda_j\doteqdot H(Y_j,s_j)$ to obtain a sequence of rescaled flows $\{M^j_t\}_{t\in (-\lambda_j^2(s_j+\frac{1}{2n}),0]}$ defined by $M^j_t\doteqdot \lambda_j(M_{\lambda_j^{-2}t+s_j}-Y_j)$. Passing to a subsequence, the final timeslices $M^j_0$ converge locally uniformly in the Hausdorff topology to a convex hypersurface $M^\infty_0=\pd\Omega^\infty_0$. 

We claim that $M^\infty_0$ splits off a line. To see this, consider the segments $\ell_i$ joining $Y_j\in \pd\Omega_{s_j}$ to $0\in \Omega_{0}\subset \Omega_{s_j}$. Passing to a subsequence, these segments converge to a ray, $\ell$, emanating from $0$. Observe that $\ell\subset \Omega_{0}$ for all $j$. Indeed, if this were not the case, then we could find a point $p\in \ell\cap \pd\Omega_{0}$. But then, since $\ell\cap\pd\Omega_{-1/2}=\emptyset$, the tangent hyperplane to $\pd\Omega_t$ parallel to $T_p\pd\Omega_0$ must travel an infinite distance in finite time, contradicting the uniform curvature bound. Now consider the segment $\ell'_i$ obtained from $\ell_i$ by reflection across the hyperplane orthogonal to $\ell$ and through $Y_j$. Since $\Omega_{s_j}$ is convex, the triangle bounded by $\ell_i$, $\ell_i'$ and $\ell$ lies in $\overline\Omega_{s_j}$ for all $j$. The claim follows, since the angle between $\ell_i$ and $\ell_i'$ goes to $\pi$ as $i\to\infty$ and $\lambda_i\ge \delta>0$.

Since (after passing to a subsequence) the convergence is smooth in $P_{\frac{\delta}{8}}$, we find that $H=1$, $\kappa_1=0$ and $\frac{\kappa_1}{H}\ge \varepsilon$ at the spacetime origin, which is absurd.
\end{proof}

Finally, we rule out pinched expanding or translating solutions.

\begin{prop}[No pinched solitons]\label{prop:translator and expander}
There exist no locally uniformly convex pinched mean curvature flow translators or expanders.
\end{prop}
\begin{proof}
We proceed as in \cite{BL17} and \cite{Ni05}. First, let $M^n\subset \R^{n+1}$, $n\ge 2$, be a locally uniformly convex, pinched mean curvature flow translator. Then $M$ satisfies
\[
H=-\inner{e}{\nu}
\]
for some $e\in \R^{n+1}\setminus\{0\}$. Observe that the vector field $V\doteqdot e^\top$ satisfies
\begin{equation}\label{eq:gradient field}
L(V)+\cd H=0\,,
\end{equation}
and
\begin{equation}\label{eq:gradient of gradient field translator}
\cd V=HL\,,
\end{equation}
where $L$ denotes the Weingarten map (see, for example, Lemma 13.32 in the book \cite{EGF} of Chow \emph{et al.}).

By Proposition \ref{prop:curvature bound at infinity}, $H$ attains its maximum at some point $o\in M$. Since $A>0$, \eqref{eq:gradient field} implies that $o$ is a zero of $V$. We claim that $V$ vanishes nowhere else. To see this, fix $X\in M\setminus\{o\}$ and let $\gamma:[0,d(X)]\to M$ be any minimizing unit speed geodesic joining $o$ to $X$, where $d$ denotes the intrinsic distance to $o$. Observe that
\ba\label{eq:distance bound}
\inner{V_X}{\gamma'(d(X))}={}&\int_0^{d(X)}\!\!\frac{d}{ds}\inner{V\circ\gamma}{\gamma'}ds=\int_0^{d(X)}\!\!\inner{\cd_{\gamma'}V}{\gamma'}ds\,.
\ea
This is positive (and hence $V_X\neq 0$) since, by \eqref{eq:gradient of gradient field translator}, $\cd V$ is positive definite. In fact, since $\vert V\vert^2=1-H^2$ and $A\ge \alpha Hg$ for some $\alpha>0$, we obtain (we will use this in a moment)
\begin{equation}\label{eq:translator estimate}
1\ge \vert V\vert \ge \alpha\int_0^{d}H^2\,ds=\alpha d-\alpha\int_0^{d}\vert V\vert^2\,ds\,.
\end{equation}

It follows that $M\setminus\{o\}$ is foliated by integral curves $\sigma:(0,\infty)\to M$ of $V$ with $\sigma(s)\to o$ as $s\to 0$. Let $\sigma$ be such a curve. By \eqref{eq:gradient field} and the pinching hypothesis,
\begin{equation}\label{eq:H estimate}
-\frac{d}{ds}\log(H\circ\sigma)=-\frac{\cd_{V}H}{H}=\frac{A(V,V)}{H}\ge \alpha\vert V\vert^2\,.
\end{equation}
Combining \eqref{eq:translator estimate} and \eqref{eq:H estimate} yields
\[
H\le H(o)\mathrm{e}^{1-\alpha d}\,.
\]

Since $M$ is convex and non-flat, it follows from the Ecker--Huisken interior estimates that $\lambda M$ converges as $\lambda\to 0$ locally uniformly in $C^\infty(\R^{n+1}\setminus\{0\})$ to a non-planar convex cone. But this violates pinching.

Now, let $M^n\subset \R^{n+1}$, $n\ge 2$, be a locally uniformly convex, pinched mean curvature flow expander. Then $M$ satisfies
\[
H=-\frac{1}{2}\inner{X}{\nu}\,.
\]
Observe that the vector field $V\doteqdot \tfrac{1}{2}X^\top$ satisfies \eqref{eq:gradient field} and
\begin{equation}\label{eq:gradient of gradient field expander}
\cd V=HL+\tfrac{1}{2}I\,.
\end{equation}
(See, for example, \cite[Lemma 10.14]{EGF}).

By Proposition \ref{prop:curvature bound at infinity}, $H$ attains its maximum at some point $o\in M$. Since $A>0$, \eqref{eq:gradient field} implies that $o$ is a zero of $V$. Arguing as in \eqref{eq:distance bound}, we find that $V$ vanishes nowhere else. In fact, we obtain
\begin{equation}\label{eq:expander estimate}
\vert V\vert \ge \frac{\alpha}{2}d
\end{equation}
for some $\alpha>0$. On the other hand,
\begin{equation}\label{eq:expander distance}
\frac{d}{ds}d\circ\sigma=\inner{V\circ \sigma}{\cd d\circ \sigma}\le\vert V\vert\,.
\end{equation}
Since \eqref{eq:gradient field} holds, \eqref{eq:H estimate} also holds (with $V=\frac{1}{2}X^\top$) on an expander. Combining \eqref{eq:H estimate}, \eqref{eq:expander estimate} and \eqref{eq:expander distance}, we conclude that
\[
H\le H(o)\mathrm{e}^{-\frac{\alpha}{4}d^2}\,.
\]
The claim now follows as before.
\end{proof}

Putting these ingredients together yields the result.

\begin{thm}[Hamilton \cite{HamiltonPinched}]
Every pinched, convex hypersurface with bounded curvature in $\R^{n+1}$, $n\geq 2$, is either a hyperplane or compact.
\end{thm}
\begin{proof}
So let $M=\pd\Omega\subset \R^{n+1}$, $n\ge 2$, be a convex hypersurface with bounded curvature which is $\alpha$-pinched for some $\alpha>0$. 

By Proposition \ref{prop:existence}, we obtain a family $\{M_t\}_{t\in[0,T)}$ of convex boundaries $M_t=\pd\Omega_t$ with $M_0=M$ which evolve by mean curvature and satisfy $\sup_{t\in[0,\sigma]}\sup_{M_t}\vert A\vert<\infty$ for all $\sigma\in[0,T)$, and either $T=\infty$, or 
\[
\limsup_{t\to T}\sup_{M_t}\vert A\vert=\infty\,.
\]
By Proposition \ref{prop:pinching preserved}, $M_t$ is $\alpha$-pinched for each $t\in(0,T)$.
By applying the strong maximum principle to the evolution equation for $H$ (see, for example, \cite[Equation (6.18)]{EGF}), we may assume that $H>0$ on $M_t$ for all $t\in(0,T)$.

\subsection*{Case 1: $T<\infty$} Since Proposition \ref{prop:curvature bound at infinity} implies that
\begin{equation}\label{eq:type III at infinity}
\limsup_{\vert X\vert\to\infty}\vert A_{(X,t)}\vert^2\le 0\;\;\text{for all}\;\; t\in(0,T)\,,
\end{equation}
the local umbilic estimate (Proposition \ref{prop:local pinching}) yields
\begin{equation}\label{eq:umbilic}
\vert \mathring A\vert
\le \varepsilon H+C(n,\alpha,\varepsilon)\Theta
\end{equation}
for every $\varepsilon>0$, where $\Theta\doteqdot \sup_{M_{0}}H<\infty$. 

Since $\vert A_{(X,t)}\vert\to 0$ as $\vert X\vert\to\infty$, we can find a sequence of times $t_j\to T$ and points $X_j\in M_{t_j}$ such that
\[
\lambda_j\doteqdot H(X_j,t_j)=\max_{t\in [0,t_j]}\max_{M_{t}}H>0\,.
\]
Translating $(X_j,t_j)$ to the spacetime origin and rescaling by $\lambda_j$ yields a sequence of mean curvature flows $\{M^j_t\}_{t\in(-\lambda_j^2t_j, 0]}$ defined by $M^j_t\doteqdot \lambda_j(M_{\lambda_j^{-2}t+t_j}-X_j)$. Since $\max_{M^j_t}H\le 1=H(0,0)$ and $\lambda_j\to\infty$ as $j\to\infty$, a subsequence of the rescaled flows converges locally uniformly in $C^\infty(\R^{n+1}\times(-\infty,0])$ to an ancient flow $\{M^\infty_t\}_{t\in(-\infty,0]}$. By \eqref{eq:umbilic}, this limit is umbilic, and hence a shrinking sphere. So $M$ is compact.

\subsection*{Case 2: $T=\infty$.} 

Suppose first that the flow is \emph{type-III}; that is, $\Lambda\doteqdot \sup_{t\in(0,\infty)}\sqrt{t}\max_{M_t}H<\infty$.
Given a sequence of times $t_j>0$ with $t_j\to\infty$, consider the rescaled flow $\{M^j_t\}_{t\in(-\lambda_j^2t_j,\infty)}$ defined by $M^j_t\doteqdot \lambda_jM_{\lambda_j^{-2}t+t_j}$, where $\lambda_j\doteqdot \frac{1}{\sqrt{t_j}}$. Since $\lambda_j^2t_j\equiv 1$ and
\bann
H_j(\cdot,t)={}&\lambda_j^{-1}H(\cdot,\lambda_j^{-2}t+t_j)\le\frac{\Lambda}{\sqrt{t+1}}\,,
\eann
we obtain a subsequential limit defined for $t\in(-1,\infty)$. Furthermore,
\[
\sqrt{t+1}\,H_\infty(\cdot,t)=\lim_{j\to\infty}\sqrt{t+1}H_j(\cdot,t)=\lim_{j\to\infty}\sqrt{\frac{t+1}{\lambda_j^2}}H\left(\cdot,\frac{t+1}{\lambda_j^2}\right)\,.
\]
By the differential Harnack inequality and the type-III hypothesis, the limit on the right exists (and is positive) independently of $t$. Thus, by the rigidity case of the differential Harnack inequality, the limit is a nontrivial expander, which violates Proposition \ref{prop:translator and expander}.

So suppose that the flow is \emph{type-IIb}; that is, $\sup_{t\in(0,\infty)}\sqrt{t}\max_{M_t}H=\infty$. For each $j$, choose $(x_j,t_j)$ such that
\[
t_j(j-t_j)H^2(x_j,t_j)=\max_{M\times [0,j]}t(j-t)H^2(\cdot,t)
\]
and set $\lambda_j\doteqdot H(x_j,t_j)$. The corresponding rescaled flows satisfy
\bann
H_j(\cdot,t_j)
\le{}&\sqrt{\frac{-\alpha_j}{t-\alpha_j}\frac{\omega_j}{\omega_j-t}}
\eann
for all $t\in(\alpha_j,\omega_j)$, where $\alpha_j\doteqdot -\lambda_j^{2}t_j$ and  $\omega_j\doteqdot \lambda_j^{2}(j-t_j)$. Since
\bann
\frac{1}{\omega_j^{-1}-\alpha_j^{-1}}
\ge{}& \frac{1}{2}\max_{M\times[0,j/2]}tH^2\,,
\eann
we find that $\alpha_j\to-\infty$ and $\omega_j\to\infty$, and hence obtain an eternal limit flow $\{M_t\}_{t\in(-\infty,\infty)}$. Since $\max H$ is attained at the spacetime origin (where it is positive), the differential Harnack inequality implies that $\{M_t\}_{t\in(-\infty,\infty)}$ evolves by translation. This violates Proposition \ref{prop:translator and expander}. 
\end{proof}

\bibliographystyle{plain}
\bibliography{../../bibliography}

\end{document}